\documentclass[a4paper,legno,11pt]{article}

\usepackage[utf8]{inputenc}
\usepackage[T1]{fontenc}
\usepackage{xspace,
	amsthm,
	amsmath,
	amssymb,
	bbm,
	tikz,
	graphicx,
	subfig,
	keyval}
\DeclareGraphicsExtensions{.png, .pdf}

\usepackage[hidelinks]{hyperref}

\usepackage{geometry}
\tikzstyle{nodino}=[circle,draw,fill,inner sep=0pt,minimum size=0.5mm]
\tikzstyle{infinito}=[circle,inner sep=0pt,minimum size=0mm]
\tikzstyle{nodo}=[circle,draw,fill,inner sep=0pt, minimum size=0.5*width("k")]
\tikzstyle{nodo_vuoto}=[circle,draw,inner sep=0pt, minimum size=0.5*width("k")]
\usetikzlibrary{graphs}
\tikzset{every loop/.style={min distance=10mm,in=300,out=240,looseness=10}}
\tikzset{place/.style={circle,thick,draw=blue!75,fill=blue!20,minimum
		size=6mm}}
\tikzset{place2/.style={circle,thick,draw=red!75,fill=red!20,minimum
		size=6mm}}

\newcommand{\f}{\frac}
\newcommand{\R}{{\mathbb R}}
\newcommand{\rr}{{\mathbb R}}

\newcommand{\zz}{{\mathbb Z}}

\newcommand{\G}{{\mathcal{G}}}
\newcommand{\udot}{\|\psi'\|_{L^2(\mathcal{G})}}
\newcommand{\uLp}{\|\psi\|_{L^p(\mathcal{G})}}
\newcommand{\uLtwo}{\|\psi\|_{L^2(\mathcal{G})}}

\newcommand{\intX}{\int_{X_{j,k}}|\psi'(t)|\,dt}
\newcommand{\intY}{\int_{Y_{i,k}}|\psi'(t)|\,dt}
\newcommand{\intZ}{\int_{Z_{i,j}}|\psi'(t)|\,dt}

\theoremstyle{plain} 
\newtheorem{thm}{Theorem}[section]

\newtheorem{prop}[thm]{Proposition} 

\theoremstyle{definition}

\theoremstyle{definition}

\theoremstyle{remark}

\author{Riccardo Adami, Simone Dovetta \\
Dipartimento di Scienze Matematiche G.L. Lagrange \\
Politecnico di Torino \\
C.so Duca Degli Abruzzi 24, 10129 Torino, Italy \\ 
{riccardo.adami@polito.it}, simone.dovetta@polito.it \\ \ \\
{\em To Gianfausto Dell'Antonio for his 85th birthday}}

\title{One-dimensional versions of three-dimensional system: Ground states for the NLS on the spatial grid}

\begin{document}
	
\maketitle

\begin{abstract}
We investigate the existence of ground states for the focusing Nonlinear Schr\"odinger Equation on the infinite three-dimensional cubic grid. We extend the result found for the analogous two-dimensional grid by proving an appropriate Sobolev inequality giving rise to a family of critical Gagliardo-Nirenberg inequalities that hold for every nonlinearity power from $10/3$ and $6$, namely, from the $L^2$-critical power for the same problem in $\R^3$ to the critical power for the same problem in $\R$. Given the Gagliardo-Nirenberg inequality, the problem of the existence of ground state can be treated as already done for the two-dimensional grid.

\end{abstract}

	\section{Introduction}
Even though {\em quantum graphs} have a local one-dimensional structure, it is possible to envisage graphs extending in different directions and looking on the large scale as planar, spatial, or even more-dimensional. Potentially, in such a way one could try to build up one-dimensional models capable to mimic the dynamics of higher dimensional systems. Such a task can be seen as 
an advanced version of the classical issue of approximating a problem in $\R^n$ through a point grid, with the advantage of reproducing the original model at a scale shorter than the distance between two adjacent points of the lattice.

At present, the research programme on graph approximation to higher dimensional structures is at its beginning and, to our knowledge, the only result concerning the NLS has been given in \cite{ADST} for the graph called {\em two-dimensional grid} (see Fig. \ref{2d}) and establishes that, concerning existence of ground states, the energy functional associated to the Schr\"odinger Equation with a power nonlinearity may reproduce both a one and a two-dimensional behaviour. In particular, it is established that all powers between the  critical ones for the NLS in $\R$ and in $\R^2$ can be considered as critical  for the NLS on the two-dimensional grid, and this feature can be described as a dimensional crossover, in accordance with the double nature of the graph: one-dimensional on the short scale, two-dimensional in the large.

\begin{figure} 
	\begin{center}
		\begin{tikzpicture}[xscale= 0.5,yscale=0.5]
		\draw[step=2,thin] (0,0) grid (8,8);
		\foreach \x in {0,2,...,8} \foreach \y in {0,2,...,8} \node at (\x,\y) [nodo] {};
		\foreach \x in {0,2,...,8}
		{\draw[dashed,thin] (\x,8.2)--(\x,9.2) (\x,-0.2)--(\x,-1.2) (-1.2,\x)--(-0.2,\x)  (8.2,\x)--(9.2,\x); }
		\end{tikzpicture}
	\end{center}
	\caption{The two-dimensional grid.}
	\label{2d}
\end{figure}
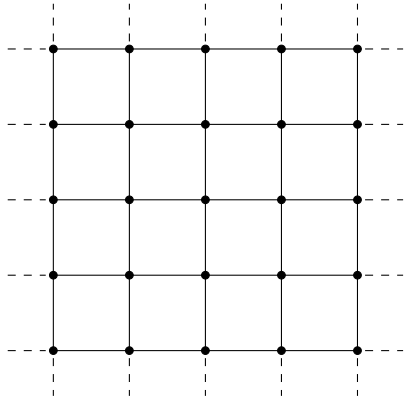

In this note we extend the analysis to the case of the three-dimensional grid (see Fig. \ref{FIG-grid}) and find an analogous result: the problem possesses both the critical powers of its analogous in $\R$ and in $\R^3$ (and, by interpolation, all intermediate powers can be considered as critical). Exploiting this fact, and following step-by-step the line of \cite{ADST}, one immediately replicates the result on existence and nonexistence of ground states given by Theorems 1.1--1.2 in \cite{ADST}.

Before stating precisely the result, let us introduce some definitions and notation. In the ordinary euclidean space $\R^n$ it is well-known that the focusing nonlinear Schr\"odinger Equation
$$ i \partial_t \psi (t,x) \ = \ - \Delta \psi (t,x) - | \psi (t,x) |^{p-2} \psi (t,x) $$
is locally well-posed in $H^1 (\R^n)$ provided that $2 \leq p \leq  2n/(n-2)$ if $n \geq 3$ and for every $p$ if $n = 1,2$. Moreover, the $L^2$-norm (whose square is also called {\em mass}) and the energy
\begin{equation} \label{energy}
\begin{split}
E (\psi(t)) \ = \ & \f 1 2 \| \psi (t) \|_{L^2 (\R^n)}^2 - \f 1 p \| \psi (t) \|_{L^p (\R^n)}^p
\\
\ = \ & T (\psi (t)) - V (\psi (t))
\end{split}
\end{equation}
are conserved, as long as the solution $\psi (t)$ exists. Through the symbols $T$ and $V$ we made clear that the first term in the energy is to be understood as kinetic energy, while the second as a nonlinear potential.

It transpires that, in order for $E$ to be lower bounded, it is necessary that every function $\psi (t)$ bearing a large
 potential $V(\psi (t))$ possesses an at least equally large kinetic energy $T(\psi (t))$. However, as the potential scales with a larger power, it is immediately seen that, regardless of the dimension $n$, the energy functional $E$ cannot be lower bounded, and in particular a global minimizer in $H^1$ cannot exist. Nonetheless, due to the conservation law of the mass, it is  meaningful to investigate whether a minimizer exists among functions sharing the same mass $\mu$ ({\em{mass constraint}}): we study then a constrained problem, more precisely the minimization  of a functional on $L^2$-spheres.

To this aim, one needs to compare $T$ and $V$, and the main tool to do that is provided 
 by the Gagliardo-Nirenberg inequalities
\begin{equation} \label{gajardo}
V \ \leq \ C T^\alpha \mu^\beta,
\end{equation}
where  $\mu$ denotes the mass of $\psi (t)$ (notice that by homogeneity $2 \alpha + 2 \beta = p$).
 The precise determination of $\alpha$ as a function of the power $p$ can be done by observing that the inequality must hold after the dilation $x \mapsto \lambda x$ and this is possible only if
$$\alpha = n \left( \frac p 4 - \frac 1 2 \right),$$ 
so that the threshold between the two behaviours is given by
\begin{equation}
\label{critical}
p^* = 2 + \frac 4 n.
\end{equation}
Plugging \eqref{gajardo} into \eqref{energy} supplies the estimate from below 
$$
E (\psi (t)) \ \geq \ T (\psi (t)) - C T^\alpha (\psi (t))
$$
where the factor $\mu^\beta$ in \eqref{gajardo} was included in the constant $C$. It becomes then clear
that if $p < p^*$ then $\alpha < 1$ and the constrained energy is lower bounded for every $\mu$. Furthermore, one can see that the infimum is actually strictly negative. On the contrary, an easy rescaling argument shows that the constrained energy is unbounded from below for every $\mu$ if $p > p^*$.

If $p= p^*$, then $\alpha = 1$ and the Gagliardo-Nirenberg inequality \eqref{gajardo} reads
\begin{equation} \label{criticalgn}
V \leq C \mu^{\frac p 2 - 1} T,
\end{equation}
so the estimate from below for the constrained energy becomes
$$ E (\psi (t)) \ \geq \ T(\psi(t)) ( 1 - C  \mu^{\frac p 2 - 1}),
$$
thus
the functional $E$ is lower bounded and nonnegative for $\mu \leq \mu^*$, where $\mu^*
= C^\frac{2}{2-p}$
is called the {\em critical mass}. Moreover it is well-known that only for $\mu = \mu^*$ the energy reaches level zero (so there are ground states), while for $\mu > \mu^*$ the functional is unbounded from below. We stress that the notion of critical mass is meaningful only when the nonlinearity under investigation is the critical one.





Analysing more closely the meaning of the critical power $p^*$, one notices that it 
enjoys two distinct features:
\begin{enumerate}
\item $p^*$ is the highest power below which there exists a negative energy ground state, regardless of the choice of $\mu$;
\item $p^*$ is the lowest power beyond which the energy functional is not lower bounded, regardless of the choice of $\mu$.
\end{enumerate}

\smallskip

As nowadays widely recognized, it is possible to extend the Nonlinear Schr\"odinger Equation from $\R^n$ to quantum graphs (\cite{acfn10,noja,pelinovsky,matrasulo}), namely connected structures made of edges and vertices. In order to gain well-posedness of the problem on graphs, one must specify the behaviour of the wave function at the vertices. Here we restrict our scope to the choice of imposing \emph{continuity} of the wave function at vertices: this is enough in order to have a well-defined energy functional formally analogous to \eqref{energy}, and can be considered somewhat natural since $H^1$-regularity, and then continuity, inside the edges is requested by the definition of the functional \eqref{energy}. For the well-posed system one immediately has conservation of mass and energy (\cite{acfn10}). Concerning the existence of ground states and stationary solutions of the NLS, the problem has been addressed on non-compact graphs made of a compact core and a certain number of halflines, both in the case of extended (\cite{AST2015,ast16,AST-cmp17,AST-ln}) and concentrated nonlinearity (\cite{DT-rxv18,ST-jde16,ST-na16,tentarelli-jmaa16}), and on compact graphs (\cite{dovetta,CDS}). 

In such settings, by a straightforward argument based on rearrangement techniques, one finds that one-dimensional Gagliardo-Nirenberg inequalities hold in every metric graph (\cite{ast16}),
so that the power $p^* = 6$ turns out to be critical for the NLS on every quantum graph, in the sense that if $p=6$, then the critical Gagliardo-Nirenberg \eqref{criticalgn} holds.

Yet, for some classes of quantum graphs the critical power turns out not to be unique. More specifically,
in \cite{ADST} it was proven that the NLS on the two-dimensional grid admits critical nonlinearities for every power between $p^*_- = 4$ and $p^*_+ = 6$. More precisely
\begin{enumerate}
\item $p^*_-$ is the highest power below which there exists a negative energy ground state, regardless of the choice of $\mu$;
\item $p^*_+$ is the lowest power over which the energy functional is not lower bounded, regardless of the choice of $\mu$.
\end{enumerate}
Now, since, in the case of the square grid, $p^*_-$ is the critical power for the NLS in $\R^2$, while $p^*_+$ is the critical power for the NLS in $\R$, the splitting of the critical power manifests itself as a dimensional transition, or {\em crossover}, between the one and the two-dimensional setting. 

The purpose of this note is to prove an analogous result for the three-dimensional grid, with the natural replacement 
$p^*_- = 10/3$, that coincides with the critical power for $\R^3$.  Of course, we are convinced that the same conclusion holds at every dimension, with $p^*_- = 2 + 4/n$, but the method of deriving the crucial estimates may vary considerably with the dimension of the grid, so that, for the sake of clarity, we treat the problem in dimension three only.

As shown before, in the domain $\R^n$ the critical power corresponds to the unique value of the nonlinearity $p^*$ in correspondence of which the power $\alpha$ in \eqref{gajardo} equals one. The uniqueness of such a value comes from the fact that $p$ uniquely determines $\alpha$ which, in turn, depends on the fact that $\R^n$ is invariant under scaling, therefore Gagliardo-Nirenberg inequality \eqref{gajardo} must be invariant too.

It is then clear that the dimensional crossover, namely, the splitting of the value of the critical power, can occur since for those more-dimensional grids the power $\alpha$ in \eqref{gajardo} cannot be uniquely determined by invariance under scaling, as grids do not enjoy such invariance! This fact makes possible to prove {\em critical Gagliardo-Nirenberg inequalities}, namely inequalities of the kind
$$
V (\psi (t)) \ \leq \ C \mu^{\frac P 2 - 1} T(\psi(t))
$$
for every power $p $ in the interval $ \left[ 2 + \frac 4 n, 6
\right]$.

Such estimates are found through appropriate
Sobolev inequalities showing that some $L^p$-norm of a function in $H^1$  can be estimated by the $L^1$-norm of its gradient. From Sobolev inequalities one can then deduce the family of Gagliardo-Nirenberg inequalities, including the critical ones.

For instance, owing to the fact that graphs are one-dimensional structures, it is easy to prove that for every non-compact graph a Sobolev inequality characteristic of $\R$ holds, namely
\begin{equation} \label{sobolev1}
\| \psi \|_\infty \ \leq \ C \| \psi ' \|_1,
\end{equation}
from which the following \emph{critical} Gagliardo-Nirenberg inequality follows:
\begin{equation}
\label{critical1}
\| \psi \|_6^6 \ \leq \ C \mu^2 \| \psi' \|_2^2.
\end{equation}
On the other hand, mimicking the standard proof for the case of $\R^2$, one gets the Sobolev inequality for the two-dimensional grid:
\begin{equation} \label{sobolev2}
\| \psi \|_2 \ \leq \ C \| \psi' \|_1
\end{equation}
from which it follows the {\em critical} Gagliardo-Nirenberg inequality
\begin{equation} \label{critical2}
\| \psi \|_4^4 \ \leq \ C \mu \| \psi' \|_2^2,
\end{equation}
that, interpolated with \eqref{critical1}, gives
$$
\| \psi \|_p^p \ \leq \ C \ \| \psi \|_2^{p-2} \| \psi' \|_2^2, \qquad \forall p \in [4,6].
$$

Here we show the validity of an analogous theory for the three-dimensional grid, namely: we first prove the Sobolev inequality
\begin{equation} \label{sobolev3}
\| \psi \|_{3/2} \ \leq \ C \| \psi' \|_1
\end{equation}
from which we obtain the three-dimensional {\em critical} Gagliardo-Nirenberg inequality
\begin{equation} \label{critical3}
\| \psi \|_{10/3}^{10/3} \ \leq \ C \mu^{2/3} \| \psi' \|_2^2.
\end{equation}
Like for the two-dimensional grid, this inequality can be interpolated with \eqref{critical1}, so finally one find that the {\em critical} Gagliardo-Nirenberg inequality
$$
\| \psi \|_p^p \ \leq \ C \ \| \psi\|_2^{p-2} \| \psi' \|_2^2, \qquad \forall p \in [10/3, 6],
$$
namely
\begin{equation} \label{last}
V (\psi (t)) \ \leq \ C \mu^{\frac p 2 - 1} T (\psi (t)),
\end{equation}
holds for $p \in \left[ \frac {10} 3, 6 \right]$.

Once proven \eqref{last}, one can follow 
Section 3 and Section 4 of \cite{ADST}, 
and prove the following
\begin{thm}[Existence of Ground States for the three dimensional grid]
	Consider the functional $E$ given by \eqref{energy} defined on the three-dimensional grid (see Fig. \ref{FIG-grid}), and call {\em ground state} any minimizer of $E$ at fixed mass $\mu$. There results
	\begin{enumerate}
		\item If $2 < p < 10/3$, then a ground state with negative energy exists for every positive $\mu$.
		\item For every $p \in [10/3 , 6]$, then there exists a {\em critical mass} $\mu_p$ such that
		\begin{enumerate}
			\item If $\mu < \mu_p$, then every function $\psi$ has positive energy and a ground state does not exist.
			\item If $p \neq \frac{10}{3}, 6$ and $\mu = \mu_p$, then there is a zero-energy ground state.
			\item If $p \neq 6$ and $\mu > \mu_p$, then there exists a negative energy ground state.
			\item If $p=6$ and $\mu\geq\mu_p$, then ground states do not exist.
		\end{enumerate}
		
	\end{enumerate}
\end{thm}

We highlight that the problem to determine whether ground states exist when $p=\frac{10}{3}$ and $\mu=\mu_{\frac{10}{3}}$ is still open.

The remainder of the note is devoted to the proof of inequality \eqref{last}.

	\begin{figure} \label{3d}
		\centering
		\begin{tikzpicture}
		[xscale= 0.6,yscale=0.6]
		\draw[step=3,thin] (0,0) grid (9,9);
		\foreach \x in {0,3,...,9} \foreach \y in {0,3,...,9} \node at (\x,\y) [nodo] {};
		
		\foreach \x in {0,3,...,9}
		{\draw[thin] (\x,9)--(\x,9.7) (\x,0)--(\x,-0.7) (-0.7,\x)--(0,\x)  (9,\x)--(9.7,\x); 
			\draw[dashed,thin] (\x,9.7)--(\x,10.2) (\x,-0.7)--(\x,-1.2) (-1.2,\x)--(-0.7,\x) (9.7,\x)--(10.2,\x);}
		
		\foreach \x in {0,3,...,9} \foreach \y in {0,3,...,9}
		{\draw[dashed,thin] (\x,\y)--(\x+2,\y+1.35);
			\draw[thin] (\x,\y)--(\x-0.5,\y-0.35);
			\draw[dashed,thin] (\x-0.5,\y-0.35)--(\x-1,\y-0.7);}
		
		\foreach \x in {1.5,4.5,...,10.5} \foreach \y in {1,4,...,10} \node at (\x,\y)
		[nodo] {};
		
		\foreach \x in {1.5,4.5,...,10.5} \foreach \y in {1,4,...,7} 
		{\draw[dashed,thin] (\x,\y)--(\x,\y+3);}
		
		\foreach \x in {1.5,4.5,...,10.5} 
		{\draw[dashed,thin] (\x,1)--(\x,0) (\x,10)--(\x,11);}
		
		\foreach \x in {1.5,4.5,...,7.5} \foreach \y in {1,4,...,10}
		{\draw[dashed,thin] (\x,\y)--(\x+3,\y);}
		
		\foreach \x in {1,4,...,10}
		{\draw[dashed,thin] (1.5,\x)--(0.5,\x) (10.5,\x)--(11.5,\x);}
		\end{tikzpicture}
		\caption{The three-dimensional grid $\G$.}
		\label{FIG-grid}
	\end{figure}
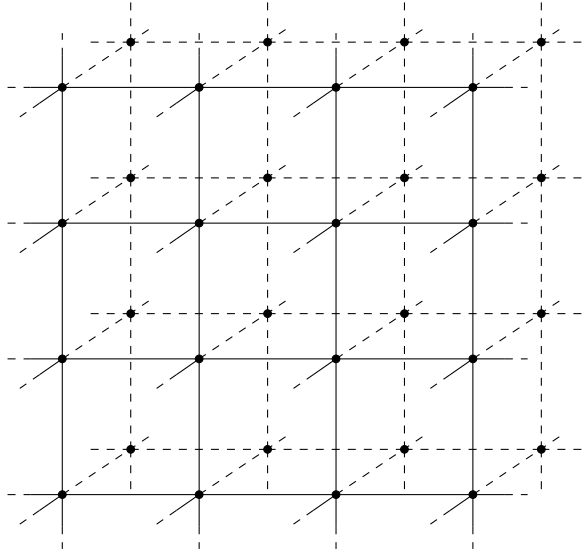

	\section{Notation}
	
	To simplify the notation, it is useful to consider the three-dimensional grid $\G$ as a subset of $\rr^3$, so that the set $V$ of vertices and the set $E$ of edges can be described as 
	
	\[
	V=\{(i\ell,j\ell,k\ell)\in\rr^3,\,(i,j,k)\in\zz^3\}\,,
	\]
	
	\[
	\begin{split}
		E=&\{[i\ell,(i+1)\ell]\times\{j\ell\}\times\{k\ell\}\subset\rr^3,\,(i,j,k)\in\zz^3\}\\
		&\cup\{\{i\ell\}\times[j\ell,(j+1)\ell]\times\{k\ell\}\subset\rr^3,\,(i,j,k)\in\zz^3\}\\
		&\cup\{\{i\ell\}\times\{j\ell\}\times[k\ell,(k+1)\ell]\subset\rr^3,\,(i,j,k)\in\zz^3\}\,.
	\end{split}
	\]
	
	\noindent Note that the set of edges naturally splits in three subsets, as above, the one whose edges are parallel to the $x-$axis, the $y-$axis and the $z-$axis, respectively. Actually, one can also interpret the whole grid $\G$ as a union of straight lines parallel to the axes, that is, defining, for every $(i,j,k)\in\zz^3$
	
	\[
	\begin{split}
	X_{j,k}:=&\rr\times\{j\ell\}\times\{k\ell\}\\
	Y_{i,k}:=&\{i\ell\}\times\rr\times\{k\ell\}\\
	Z_{i,j}:=&\{i\ell\}\times\{j\ell\}\times\rr
	\end{split}
	\]
	
	\noindent then we can write
	
	\[
	\G=\Big(\bigcup_{(j,k)\in\zz^2}X_{j,k}\Big)\cup\Big(\bigcup_{(i,k)\in\zz^2}Y_{i,k}\Big)\cup\Big(\bigcup_{(i,j)\in\zz^2}Z_{i,j}\Big)\,.
	\]
	
	\noindent Furthermore, for specific subsets of these lines we introduce the symbols
	
	\begin{equation}
	\label{EQ-notation}
	\begin{split}
		X_{j,k}(a,b):=&[a,b]\times\{j\ell\}\times\{k\ell\}\\ Y_{i,k}(a,b):=&\{i\ell\}\times[a,b]\times\{k\ell\}\\ 
		Z_{i,j}(a,b):=&\{i\ell\}\times\{j\ell\}\times[a,b]
	\end{split}
	\end{equation}
	
	\noindent where $a,b\in\rr\cup\{\pm\infty\}$.
	
	In terms of this notation, the definition of the functional spaces $L^p(\G)$ and $H^1(\G)$ is then straightforward, given by
	
	\[
	\begin{split}
		\uLp^p:=&\sum_{(j,k)\in\zz^2}\| \psi \|_{L^p(X_{j,k})}^p+\sum_{(i,k)\in\zz^2}\|\psi\|_{L^p(Y_{j,k})}^p+\sum_{(i,j)\in\zz^2}\|\psi\|_{L^p(Z_{i,j})}^p\\
		=&\sum_{(j,k)\in\zz^2}\int_{X_{j,k}}|\psi(x)|^p\,dx+\sum_{(i,k)\in\zz^2}\int_{Y_{i,k}}|\psi(x)|^p\,dx+\sum_{(i,j)\in\zz^2}\int_{Z_{i,j}}|\psi(x)|^p\,dx
	\end{split}
	\]
	
	\noindent and
	
	\[
	\| \psi \|_{H^1(\G)}^2:=\uLtwo^2+\udot^2\,,
	\]

	\noindent where we denoted by $H^1(\G)$ the space of functions whose restriction to every line $X_{j,k},\,Y_{i,k},\,Z_{i,j}$ belongs to $H^1(\rr)$ and that are continuous at every vertex of $\G$. A similar definition is stated for the space $W^{1,1}(\G)$, that is the space of of functions that are $H^1(\rr)$ when restricted to every line $X_{j,k},\,Y_{i,k},\,Z_{i,j}$ and that are continuous at every vertex.

	
	\section{Sobolev and Gagliardo-Nirenberg inequalities on $\G$}
	
	Throughout this section, we derive some inequalities that play a key-role allowing or preventing the existence of ground states on $\G$.
	
	First of all, we recall that, being $\G$ a non-compact metric graph, the following one-dimensional Gagliardo-Nirenberg inequality holds (a general proof of it, based on the use of Poly\'a-Szeg\"o inequality for symmetric rearrangements, can be found in \cite{AST JFA}).

	\begin{prop}
		\label{PROP-GN 1D}
		
		For every $p\in[2,\infty)$, there exists a constant $C_1:=C_1(p)>0$ such that
		
		\begin{equation}
		\label{EQ-GN 1D}
		\uLp^p\leq C_1\uLtwo^{\frac{p}{2}+1}\udot^{\frac{p}{2}-1}\qquad\forall u\in H^1(\G)\,.
		\end{equation}
		
		\noindent Moreover, 
		
		\begin{equation}
			\label{EQ-GN infty}
			\|\psi\|_{L^\infty(\G)}^2\leq\uLtwo^{1/2}\udot^{1/2}\,.
		\end{equation}
	\end{prop}

	The previous inequalities rely on the one-dimensional nature of the grid, and actually holds for every non-compact metric graphs. On the other hand, it is possible to exploit the three-dimensional structure of $\G$, and this leads to the following Sobolev inequality which holds only for the grid we are dealing with here.
	
	\begin{figure}[t]
		\centering
		\subfloat[][]{
			\begin{tikzpicture}
			[xscale= 0.4,yscale=0.4]
			\draw[step=3,thin] (-3,3) grid (9,9);
			\foreach \x in {-3,0,...,9} \foreach \y in {3,6,...,9} \node at (\x,\y) [nodo] {};
			
			\foreach \x in {-3,0,...,9} \foreach \y in {3,6,9}
			{\draw[thin] (\x,9)--(\x,9.7) (\x,3)--(\x,2.3)   (9,\y)--(9.7,\y); 
				\draw[dashed,thin] (\x,9.7)--(\x,10.2) (\x,2.3)--(\x,1.8) (9.7,\y)--(10.2,\y);}
			
			\foreach \x in {-3,0,...,9} \foreach \y in {3,6,...,9}
			{\draw[dashed,thin] (\x,\y)--(\x+2,\y+1.35);
				\draw[thin] (\x,\y)--(\x-0.5,\y-0.35);
				\draw[dashed,thin] (\x-0.5,\y-0.35)--(\x-1,\y-0.7);}
			
			\foreach \x in {-1.5,1.5,...,10.5} \foreach \y in {4,7,...,10} \node at (\x,\y)
			[nodo] {};
			
			\foreach \x in {-1.5,1.5,...,10.5} \foreach \y in {4,7} 
			{\draw[dashed,thin] (\x,\y)--(\x,\y+3);}
			
			\foreach \x in {-1.5,1.5,...,10.5} 
			{\draw[dashed,thin] (\x,4)--(\x,3) (\x,10)--(\x,11);}
			
			\foreach \x in {-1.5,1.5,...,7.5} \foreach \y in {4,7,10}
			{\draw[dashed,thin] (\x,\y)--(\x+3,\y);}
			
			\foreach \x in {4,7,10}
			{\draw[dashed,thin] (-1.5,\x)--(-2.5,\x) (10.5,\x)--(11.5,\x);}

			\foreach \y in {3,6,...,9} 
			{\node at (-3,\y) [nodo] {};
				\draw[thin] (-3,\y)--(0,\y);
				\draw[thin] (-3,\y)--(-3.7,\y);
				\draw[dashed,thin] (-3.7,\y)--(-4.2,\y);}
			
			\draw[line width=0.8 mm] (-3.7,6)--(4.2,6);
			\draw[dashed,line width=0.8 mm] (-4.2,6)--(-3.7,6);
			
			\node at (-5.2,6) {$\scriptstyle-\infty$};
			\node at (4,5.5) {$x$};
			\node at (-3.9,6.7) {$X_{j,k}$};
			\end{tikzpicture}}\qquad
		\subfloat[][]{
			\begin{tikzpicture}
			[xscale= 0.4,yscale=0.4]
			\draw[step=3,thin] (0,-3) grid (9,9);
			\foreach \x in {0,3,...,9} \foreach \y in {-3,0,...,9} \node at (\x,\y) [nodo] {};
			
			\foreach \x in {0,3,...,9} \foreach \y in {-3,0,...,9}
			{\draw[thin] (\x,9)--(\x,9.7) (\x,-3)--(\x,-3.7)   (9,\y)--(9.7,\y); 
				\draw[dashed,thin] (\x,9.7)--(\x,10.2) (\x,-3.7)--(\x,-4.2) (9.7,\y)--(10.2,\y);}
			
			\foreach \x in {0,3,...,9} \foreach \y in {-3,0,...,9}
			{\draw[dashed,thin] (\x,\y)--(\x+2,\y+1.35);
				\draw[thin] (\x,\y)--(\x-0.5,\y-0.35);
				\draw[dashed,thin] (\x-0.5,\y-0.35)--(\x-1,\y-0.7);}
			
			\foreach \x in {1.5,4.5,...,10.5} \foreach \y in {-2,1,...,10} \node at (\x,\y)
			[nodo] {};
			
			\foreach \x in {1.5,4.5,...,10.5} \foreach \y in {-2,1,...,7} 
			{\draw[dashed,thin] (\x,\y)--(\x,\y+3);}
			
			\foreach \x in {1.5,4.5,...,10.5} 
			{\draw[dashed,thin] (\x,4)--(\x,3) (\x,10)--(\x,11);}
			
			\foreach \x in {1.5,4.5,7.5} \foreach \y in {-2,1,...,10}
			{\draw[dashed,thin] (\x,\y)--(\x+3,\y);}

			\foreach \y in {-3,0,...,9} 
			{\draw[thin] (-0.7,\y)--(0,\y);
				\draw[dashed,thin] (-0.7,\y)--(-1.2,\y);}
			
			\foreach \y in {-2,1,...,10}
			{\draw[dashed,thin] (0.5,\y)--(1.5,\y);
				\draw[dashed,thin] (10.5,\y)--(11.5,\y);}
			
			\draw[line width=0.8 mm] (3,6)--(4.2,6);
			\draw[line width=0.8 mm] (3,6)--(3,-3.7);
			\draw[dashed,line width=0.8 mm] (3,-3.7)--(3,-4.2);
			
			\node at (4,5.5) {$x$};
			\node at (2.8,-4.6) {$\scriptstyle-\infty$};
			\node at (4,-3.7) {$Y_{i,k}$};
			
			\end{tikzpicture}}\\
		\subfloat[][]{
		\begin{tikzpicture}
[xscale= 0.6,yscale=0.4]
\draw[step=3,thin] (0,3) grid (6,9);
\foreach \x in {0,3,6} \foreach \y in {3,6,...,9} \node at (\x,\y) [nodo] {};

\foreach \x in {0,3,6} \foreach \y in {3,6,9}
{\draw[thin] (\x,9)--(\x,9.7) (\x,3)--(\x,2.3)   (6,\y)--(6.7,\y); 
	\draw[dashed,thin] (\x,9.7)--(\x,10.2) (\x,2.3)--(\x,1.8) (6.7,\y)--(7.2,\y);}

\foreach \x in {0,3,...,6} \foreach \y in {3,6,...,9}
{\draw[dashed,thin] (\x,\y)--(\x+6,\y+4.05);
	\draw[thin] (\x,\y)--(\x-0.5,\y-0.35);
	\draw[dashed,thin] (\x-0.5,\y-0.35)--(\x-1,\y-0.7);}

\foreach \x in {1.8,4.8,7.8} \foreach \y in {4.2,7.2,...,10.2} \node at (\x,\y)
[nodo] {};

\foreach \x in {1.8,4.8,7.8} \foreach \y in {4.2,7.2} 
{\draw[dashed,thin] (\x,\y)--(\x,\y+3);}

\foreach \x in {1.8,4.8,7.8} 
{\draw[dashed,thin] (\x,4.2)--(\x,3.2) (\x,10.2)--(\x,11.2);}

\foreach \x in {1.8,4.8} \foreach \y in {4.2,7.2,10.2}
{\draw[dashed,thin] (\x,\y)--(\x+3,\y) (7.8,\y)--(8.8,\y);}

\foreach \x in {4.2,7.2,10.2}
{\draw[dashed,thin] (1.8,\x)--(0.8,\x) (7.8,\x)--(8.8,\x);}

\foreach \x in {3.7,6.7,9.7} \foreach \y in {5.5,8.5,...,11.5} \node at (\x,\y) [nodo] {};

\foreach \x in {3.7,6.7,9.7} \foreach \y in {5.5,8.5}
{\draw[dashed,thin] (\x,\y)--(\x,\y+3);}

\foreach \x in {3.7,6.7} \foreach \y in {5.5,8.5,11.5}
{\draw[dashed,thin] (\x,\y)--(\x+3,\y);}

\foreach \y in {5.5,8.5,11.5}
{\draw[dashed,thin] (3.7,\y)--(2.7,\y) (9.7,\y)--(10.7,\y);}

\foreach \x in {3.7,6.7,9.7}
{\draw[dashed,thin] (\x,11.5)--(\x,12.5) (\x,5.5)--(\x,4.5);}

\foreach \x in {5.5,8.5,11.5} \foreach \y in {6.7,9.7,12.7} \node at (\x,\y) [nodo] {};

\foreach \x in {5.5,8.5} \foreach \y in {6.7,9.7,12.7}
{\draw[dashed,thin] (\x,\y)--(\x+3,\y);}

\foreach \x in {5.5,8.5,11.5} \foreach \y in {6.7,9.7}
{\draw[dashed,thin] (\x,\y)--(\x,\y+3) (\x,12.7)--(\x,13.7) (\x,6.7)--(\x,5.7);}

\foreach \y in {6.7,9.7,12.7}
{\draw[dashed,thin] (5.5,\y)--(4.5,\y) (11.5,\y)--(12.5,\y);}

\foreach \y in {3,6,...,9} 
{\draw[thin] (0,\y)--(-0.7,\y);
	\draw[dashed,thin] (-0.7,\y)--(-1.2,\y);}

\draw[line width=0.8 mm] (3,6)--(4.2,6);
\draw[line width=0.8 mm] (3,6)--(8.5,9.7);
\draw[dashed, line width=0.8 mm] (8.5,9.7)--(9,10.05);
)

\node at (2.4,6.5) {$Z_{i,j}$};
\node at (9.2,10.3) {$\scriptstyle-\infty$};
\node at (4.3,6.4) {$x$};

\end{tikzpicture}}
		\caption{The paths in the proof of Proposition \ref{PROP-GN 3D}: from $-\infty$ to $x$ along $X_{j,k}$ (a), along $Y_{i,k}$ (b) and along $Z_{i,j}$ (c).}
		\label{FIG-paths}
	\end{figure}
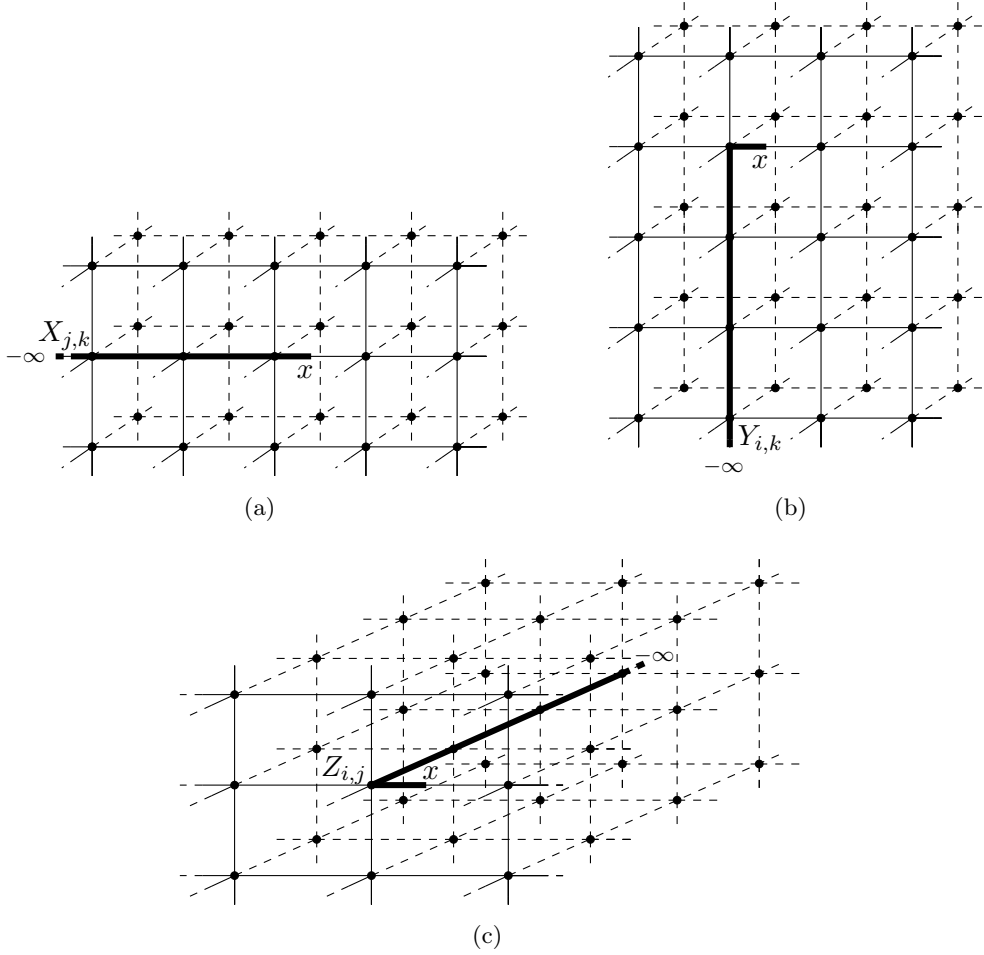

	\begin{prop}
		\label{PROP-sobolev}
		 For every $u\in W^{1,1}(\G)$, it holds
		 
		 \begin{equation}
		 	\label{EQ-sobolev}
		 	\|\psi\|_{L^{3/2}(\G)}\leq C\|\psi'\|_{L^1(\G)}\,.
		 \end{equation}
	\end{prop}
	
	\noindent with $C>0$ depending only on the edge length $\ell$.
	

\begin{proof}
		
	We adapt a standard argument used to prove the Sobolev inequalities in $\rr^n$. Let $u\in C_0^\infty(\G)$, then
	
	\begin{equation}
		\label{EQ-norm 3/2}
		\|\psi\|_{L^{3/2}(\G)}^{3/2}=\sum_{(j,k)\in\zz^2}\int_{X_{j,k}}|\psi(x)|^{3/2}\,dx+\sum_{(i,k)\in\zz^2}\int_{Y_{i,k}}|\psi(x)|^{3/2}\,dx+\sum_{(i,j)\in\zz^2}\int_{Z_{i,j}}|\psi(x)|^{3/2}\,dx\,.
	\end{equation}

	\noindent We first evaluate $\int_{X_{j,k}}|\psi(x)|^{3/2}\,dx$. Let us consider a point $x\in\G$ located on $X_{j,k}$; by the natural embedding of $\G$ in $\rr^3$, $x$ corresponds to the point $(\xi,j\ell,k\ell)\in\rr^3$. We denote by $h$ the unique integer such that $h\ell\leq\xi\leq(h+1)\ell$. Note that $h$ is a piecewise constant function of $\xi$.


	We now estimate the values of $\psi(x)$ in three ways: first, by travelling on $X_{j,k}$ from $-\infty$ to $\xi$; second, by travelling on $Y_{h,k}$ from $-\infty$ to $j\ell$ and then on $X_{j,k}$ from $h\ell$ to $\xi$; third, by travelling on $Z_{h,j}$ from $-\infty$ to $k\ell$ and then on $X_{j,k}$ from $h\ell$ to $\xi$ (see Figure \ref{FIG-paths}). Using notation in \eqref{EQ-notation}, we have  
	
	\begin{equation}
		\label{EQ-u on X}
		\psi(x)=\int_{X_{j,k}(-\infty,\xi)}\psi'(t)\,dt
	\end{equation}
	
	\begin{equation}
		\label{EQ-u on Y}
		\psi(x)=\int_{Y_{h,k}(-\infty,j\ell)}\psi'(t)\,dt+\int_{X_{j,k}(h\ell,\xi)}\psi'(t)\,dt
	\end{equation}
	
	\begin{equation}
		\label{EQ-u on Z}
		\psi(x)=\int_{Z_{h,j}(-\infty,k\ell)}
\psi'(t)\,dt+\int_{X_{j,k}(h\ell,\xi)}
\psi'(t)\,dt\,.
	\end{equation}

	\noindent Therefore, we get
	
	\begin{equation}
		\label{EQ-sob1}
		\begin{split}
			|\psi(x)|^{3/2}=\sqrt{\Big|\int_{X_{j,k}(-\infty,\xi)}\psi'(t)\,dt\Big|}\,\cdot&\,\sqrt{\Big|\int_{Y_{h,k}(-\infty,j\ell)}\psi'(t)\,dt+\int_{X_{j,k}(h\ell,\xi)}\psi'(t)\,dt\Big|}\\
			\cdot&\,\sqrt{\Big|\int_{Z_{h,j}(-\infty,k\ell)}
\psi'(t)\,dt+\int_{X_{j,k}(h\ell,\xi)}\psi'(t)\,dt\Big|}\\
			\leq\sqrt{\int_{X_{j,k}}|\psi'(t)|\,dt}\,\cdot&\,\sqrt{\int_{Y_{h,k}}|\psi'(t)|\,dt+\int_{X_{j,k}(h\ell,(h+1)\ell)}|\psi'(t)|\,dt}\\
			\,\cdot\,&\sqrt{\int_{Z_{h,j}}|\psi'(t)|\,dt+\int_{X_{j,k}(h\ell,(h+1)\ell)}|\psi'(t)|\,dt}\\
			\leq\sqrt{\int_{X_{j,k}}|\psi'(t)|\,dt}\,\cdot&\,\Big(\sqrt{\int_{Y_{h,k}}|\psi'(t)|\,dt}+\sqrt{\int_{X_{j,k}(h\ell,(h+1)\ell)}|\psi'(t)|\,dt}\Big)\\
			\cdot&\,\Big(\sqrt{\int_{Z_{h,j}}|\psi'(t)|\,dt}+\sqrt{\int_{X_{j,k}(h\ell,(h+1)\ell)}|\psi'(t)|\,dt}\Big)\,.
		\end{split}
	\end{equation}

\noindent Since $\int_{Y_{h,k}}|\psi'(t)|\,dt,\,\int_{Z_{h,j}}|\psi'(t)|\,dt$ and $\int_{X_{j,k}(h\ell,(h+1)\ell)}|\psi'(t)|\,dt$ are piecewise constant functions of $x$, integrating on $X_{j,k}$, summing over $(j,k)\in\zz^2$ and computing the products, we find
	

	\[
	\begin{split}
	\sum_{(j,k)\in\zz^2}\int_{X_{j,k}}|\psi(x)|^{3/2}\,dx\leq\ell&\sum_{(j,k)\in\zz^2}\sqrt{\intX}\cdot\Big(\sum_{i\in\zz}\sqrt{\intY}\sqrt{\intZ}\\
	+&\sum_{i\in\zz}\sqrt{\intY}\sqrt{\int_{X_{j,k}(i\ell,(i+1)\ell)}|\psi'(t)|\,dt}\\
	+&\sum_{i\in\zz}\sqrt{\intZ}\sqrt{\int_{X_{j,k}(i\ell,(i+1)\ell)}|\psi'(t)|\,dt}\\
	+&\sum_{i\in\zz}\int_{X_{j,k}(i\ell,(i+1)\ell)}|\psi'(t)|\,dt\Big)\,.
	\end{split}
	\]

	\noindent Now, it is immediate to see that
	
	\begin{equation}
		\label{eq-stima1}
		\sum_{(j,k)\in\zz^2}\sqrt{\intX}\sum_{i\in\zz}\int_{X_{j,k}(i\ell,(i+1)\ell)}|\psi'(t)|\,dt=\sum_{(j,k)\in\zz^2}\Big(\intX\Big)^{3/2}\leq\|\psi'\|_{L^1(\G)}^{3/2}\,.
	\end{equation} 
	 
	\noindent Furthermore, by Schwarz inequality

	\begin{equation}
	\label{eq-stima2}
	\begin{split}
		\sum_{(j,k)\in\zz^2}\sqrt{\intX}&\sum_{i\in\zz}\sqrt{\intY}\sqrt{\int_{X_{j,k}(i\ell,(i+1)\ell)}|\psi'(t)|\,dt}\\
		\leq&\sum_{(j,k)\in\zz^2}\sqrt{\intX}\sqrt{\sum_{i\in\zz}\intY}\sqrt{\sum_{i\in\zz}\int_{X_{j,k}(i\ell,(i+1)\ell)}|\psi'(t)|\,dt}\\
		=&\sum_{(j,k)\in\zz^2}\sqrt{\intY}\intX\leq\|\psi'\|_{L^1(\G)}^{3/2}
	\end{split}
	\end{equation} 
	
	\noindent and the same holds for the term $\sum_{(j,k)\in\zz^2}\sqrt{\intX}\sum_{i\in\zz}\sqrt{\intZ}\sqrt{\int_{X_{j,k}(i\ell,(i+1)\ell)}|\psi'(t)|\,dt}$. Finally,

	\begin{equation}
		\label{eq-stima3}
		\begin{split}
		\sum_{(j,k)\in\zz^2}\sqrt{\intX}&\sum_{i\in\zz}\sqrt{\intY}\sqrt{\intZ}\\
		\leq&\sum_{(j,k)\in\zz^2}\sqrt{\intX}\sqrt{\sum_{i\in\zz}\intY}\sqrt{\sum_{i\in\zz}\intZ}\\
		\leq&\sum_{k\in\zz}\sqrt{\sum_{i\in\zz}\intY}\sqrt{\sum_{j\in\zz}\intX}\sqrt{\sum_{(i,j)\in\zz^2}\intZ}\\
		\leq & \sqrt{\|\psi'\|_{L^1(\G)}}\sqrt{\sum_{(i,k)\in\zz^2}\intY}\sqrt{\sum_{(j,k)\in\zz^2}\intX}\leq\|\psi'\|_{L^1(\G)}^{3/2}
		\end{split}
	\end{equation}

\noindent where inequalities are justified by subsequent application of Schwarz inequality. Thus, combining \eqref{eq-stima1}, \eqref{eq-stima2}, \eqref{eq-stima3}, we end up with
	 
	\noindent 
	\[
	\sum_{(j,k)\in\zz^2}\int_{X_{j,k}}|\psi(x)|^{3/2}\leq4\ell\|\psi'\|_{L^1(\G)}\sum_{(j,k)\in\zz^2}\sqrt{\int_{X_{j,k}}|\psi'(t)|\,dt}\leq 4\ell\|\psi'\|_{L^1(\G)}^{3/2}\,.
	\]
	
	\noindent The same estimate holds for the restriction of $u$ to the union of both $Y_{i,k}$ and $Z_{i,j}$, so that
	
	\[
	\|\psi\|_{L^{3/2}(\G)}^{3/2}\leq12\ell\|\psi'\|_{L^1(\G)}^{3/2}\,.
	\]
	
	\noindent By a standard density argument, this inequality holds for every function in $W^{1,1}(\G)$ and the proof is complete.

\end{proof}

	Thanks to \eqref{EQ-sobolev}, we can now prove that also a family of three-dimensional Gagliardo-Nirenberg inequalities holds true on $\G$.
	

	\begin{prop}
		\label{PROP-GN 3D}
		For every $p\in[2,6]$, there exists a constant $C_3=C_3(p)>0$ such that
		
		\begin{equation}
			\label{EQ-GN 3D}
			\| \psi \|_{L^p(\G)}^p \ \leq \ C_3
\| \psi \|_{L^2(\G)}^{3-\frac{p}{2}}\udot^{\frac{3p}{2}-3}\qquad\forall \psi \in H^1(\G)\,.
		\end{equation}
	\end{prop}


\begin{proof}
		
		When $p=2$, inequality \eqref{EQ-GN 3D} is trivially true. 
		
		When $p=6$, the Sobolev inequality \eqref{EQ-sobolev} applied to $|\psi|^4$ (and properly modifying the constant $C$) yields
		
		\[
		\begin{split}
			\int_\G|\psi|^6\,dx=&\int_\G(|\psi|^4)^{3/2}\,dx\leq C_\ell^{3/2}\Big(\int_\G(|\psi|^4)'\,dx\Big)^{3/2}\leq C_\ell^{3/2}\Big(\int_\G4|\psi|^3|\psi'|\Big)^{3/2}\\
			\leq&C\Big(\int_\G|\psi|^6\,dx\Big)^{3/4}\Big(\int_\G|\psi'|^2\,dx\Big)^{3/4}\,,
		\end{split}
		\]

		\noindent that is
		
		\begin{equation}
			\label{EQ-u 6 con psi' 2}
			\| \psi \|_{L^6(\G)}^6=\int_\G|\psi|^6\,dx\leq C\Big(\int_\G|\psi'|^2\,dx\Big)^3=C\udot^{6}\,,
		\end{equation}

		\noindent i.e., \eqref{EQ-GN 3D} for $p=6$.
		
		When $p\in (2,6)$, writing $p=2t+6(1-t)$, for $t\in(0,1)$, we have
		
		\[
		\begin{split}
		\uLp^p=\int_\G|\psi|^p\,dx=&\int_\G|\psi|^{2t}|\psi|^{6(1-t)}\,dx\leq \Big(\int_\G|\psi|^2\,dx\Big)^t\Big(\int_\G|\psi|^6\,dx\Big)^{1-t}\\
		\leq&C^{1-t}\uLtwo^{2t}\udot^{6(1-t)}=C'\uLtwo^{3-\frac{p}{2}}\udot^{\frac{3p}{2}-3}
		\end{split}
		\]
		
		\noindent and we conclude.
		
	\end{proof}
	
	The co-existence of both one-dimensional and three-dimensional Gagliardo-Nirenberg inequality reflects into the appearance of a brand new class of such inequalities, that become relevant in the following.
	
	\begin{prop}
		\label{PROP-GN interp}
		For every $p\in\Big[\frac{10}{3},6\Big]$, there exists a constant $C_p=C_p(p)>0$ such that 
		\begin{equation}
			\label{EQ-GN interp}
			\uLp^p\leq C_p\uLtwo^{p-2}\udot^2\qquad\forall \psi
\in H^1(\G)\,.
		\end{equation}
	\end{prop}

	\begin{proof}
		
	Notice that \eqref{EQ-GN interp} reduces to \eqref{EQ-GN 3D} when $p=\frac{10}{3}$, whereas it coincides with \eqref{EQ-GN 1D} when $p=6$. Then, for every $p\in\Big(\frac{10}{3},6\Big)$, \eqref{EQ-GN interp} follows by interpolation as in the final part of the previous proof.
		
	\end{proof}

\end{document}